\documentclass[oneside, 11pt]{amsart}

\usepackage{amsmath, amssymb, enumerate} 
\usepackage[margin=1 in, includefoot, footskip=30pt]{geometry}
\usepackage[foot]{amsaddr}
\usepackage{graphicx, xcolor,subfigure,tablefootnote,threeparttable}
\usepackage{algorithm, algorithmic}
\usepackage{booktabs}
\usepackage{hyperref}
\newcounter{Hequation}

\makeatletter
\g@addto@macro\equation{\stepcounter{Hequation}}
\makeatother
\usepackage{setspace}
\usepackage{pifont}
\usepackage{amsfonts ,mathtools, amsthm,amssymb}
\usepackage{amsmath}
\usepackage{array}
\usepackage{tikz}
\usepackage{verbatim}
\usepackage{url}
\usepackage{float}

\theoremstyle{definition}

\newtheorem{prop}{Proposition}
\newtheorem{lem}{Lemma}

\newtheorem*{rmk}{Remark}
\renewcommand{\l}{\left}
\renewcommand{\r}{\right}

\newcommand{\stack}[2]{\begin{subarray}{c} #1 \\ #2 \end{subarray}}
\newcommand{\answered}[1]{ }

\title{\large A Subgradient Approach for Constrained Binary Optimization via Quantum Adiabatic Evolution}
\date{\today}


\author{Sahar Karimi}
\author{Pooya Ronagh}

\address[Sahar Karimi, Pooya Ronagh]{1QB Information Technologies (1QBit), 458-550 Burrard Street, Vancouver, BC, V6C 2B5, Canada}

\email[Sahar Karimi]{sahar.karimi@1qbit.com}
\email[Pooya Ronagh]{pooya.ronagh@1qbit.com}

\begin{document}
\maketitle
\onehalfspacing

\begin{abstract}
An earlier work \cite{Ronagh15} proposes a method for solving the Lagrangian dual of a constrained binary quadratic programming problem via quantum adiabatic evolution using an outer approximation method. This should be an efficient prescription for solving the Lagrangian dual problem in the presence of an ideally noise-free quantum adiabatic system. However, current implementations of quantum annealing systems demand methods that are efficient at handling possible sources of noise. In this paper, we consider a subgradient method for finding an optimal primal-dual pair for the Lagrangian dual of a constrained binary polynomial programming problem. We then study the quadratic stable set (QSS) problem as a case study. We see that this method applied to the QSS problem can be viewed as an instance-dependent penalty-term approach that avoids large penalty coefficients. Finally, we report our experimental results of using the D-Wave 2X quantum annealer and conclude that our approach helps this quantum processor to succeed more often in solving these problems compared to the usual penalty-term approaches. 
\end{abstract}

\section{Introduction} 

Quantum annealing hardware has been employed to solve \emph{unconstrained binary quadratic programming} (UBQP) problems \cite{McGeoch2013}. Motivated by real-world applications, several studies have focused on extending the capabilities of this hardware to solve more-general optimization problems \cite{Venturelli2015, Gili, Dominic2015, Zick}. In practice, formulations of real-world problems contain large lists of constraints, and a technique used frequently in the quantum annealing literature is the penalizing of these constraints in a single, unconstrained, objective function of several binary variables. 

In the penalty methods, the penalty coefficients are assigned to be larger than a threshold, which we will call the \emph{theoretical penalty bound}. The theoretical penalty bounds are computed such that the resulting unconstrained optimization problem is equivalent to the constrained formulation for all instances of the problem. It is not generally trivial to find a tight theoretical penalty bound for a given optimization problem. Moreover, a penalty bound that is sufficiently large for a given instance of a problem is often much smaller than the theoretical penalty bound.     

One may view quantum annealing as a physical implementation of a heuristic evolution. As such, the appearance of terms that are different in orders of magnitude may create objective functions that are difficult for the quantum annealer to explore. The several sources of noise (elaborated upon in Section \ref{sec:prelim-aqc}) of quantum annealing hardware is an added reason of inconvenience of such objective functions. Therefore, the theoretical penalty bounds are generally not suitable for the application of quantum annealers. Finally, it is important to note that penalty methods cannot efficiently handle inequality constraints. 

In \cite{Ronagh15}, Ronagh et al. explored application of outer approximation method for solving constrained problems using quantum adiabatic computation. In the present paper, we explore an alternative approach to solving a {\em constrained binary polynomial programming} (CBPP) problem stated formally as 
\begin{equation} \label{BPP_def} \tag{CBPP}
\begin{array}{llcl}
	\max &  f(x),& \\
	\text{s.t.} & g_i(x) = \mathbf 0&\text{for}& i=1, \ldots, m,\\
	 & h_j(x) \leq \mathbf 0 &\text{for} & j=1, \ldots, p,\\
	 & x \in \mathbb B^n,& 
\end{array}
\end{equation}
where $f$, $g_i$ for $i=1,\ldots, m$,  and $h_j$ for $j=1,\ldots, p$ are a finite number of polynomials on $\mathbb B^n= \{0, 1\}^{n}$. 

Note that, since for any three binary variables $x, y, z \in \mathbb B$, the locus 
\begin{equation} \label{eq:degree-reduction}
\arg \max 2z (x+ y) - 3z - xy
\end{equation}
is identical to the vanishing locus of $xy -z = 0$ on $\mathbb B^3$, we may without loss of generality assume that all CBPP problems have at least an equivalent representation 
\begin{equation} \label{CBQP_def} \tag{CBQP}
\begin{array}{llcl}
	\max &  f(x),& \\
	\text{s.t.} & g_i(x) = \mathbf 0&\text{for}& i=1, \ldots, m,\\
	 & h_j(x) \leq \mathbf 0 &\text{for} & j=1, \ldots, p,\\
	 & x \in \mathbb B^n,& 
\end{array}
\end{equation}
where $f$, and all $g_i$ and $h_j$ are of degree at most two. Note that any quadratic function of the form $x^t A x + b^tx + c$ on $x\in \mathbb B^n$ can be written as $x^t  \l(A+ \text{diag}(b) \r) x +  c$ by the fact that  $x_i^2 = x_i$, where $\text{diag($b$)}$ is a diagonal matrix with entries of vector $b$ on the diagonal.

A subclass of \emph{constrained binary quadratic programming} (CBQP) problems is the {\em quadratic stable set} (QSS) problem. Let $G=(V,E)$ be a graph with vertex set $V$ and edge set $E$. $\mathcal S \subseteq V$ is a stable set of $G$ if the subset of edges with both endpoints in $\mathcal S$ is empty. Let $W$ define  a matrix of weights between each pair of vertices, and $A$ stand for the adjacency matrix of graph $G$. The following is a formal presentation of the QSS problem: 
\begin{equation} \label{QSS_def} \tag{QSS}
\begin{array}{llcl}
	\max &  x^tWx,& \\
	\text{s.t.} & x^tAx=0,\\
	 & x \in \mathbb B^n.& 
\end{array}
\end{equation}
Unlike its well-known linear counterpart, that is, the stable set problem (also known as the maximum independent set problem), the QSS problem is more contemporary and has been addressed to a lesser extent in the literature (see  \cite{QSS_FuriniTraversi}, \cite{BiqCrunch_s}, and references therein). 

In this paper, we present a method for solving the Lagrangian dual of a CBPP problem using a subgradient descent approach and quantum annealing. Once applied to \eqref{QSS_def}, this method terminates in strong duality, hence solving \eqref{QSS_def} to optimality without the need for a branch-and-bound scheme. In fact, the results presented in this paper apply to a more general form of the QSS problem, in which $A$ is not necessarily an adjacency matrix and entries of $A$ can take any non-negative values, that is, a problem similar to formulation \eqref{QSS_def} in which $A\ge \mathbf 0$. We refer to this variant of the QSS problem as the {\em generalized quadratic stable set} (GQSS) problem.

The paper is organized as follows. In Section 2, we review the quantum adiabatic approach to solving \emph{unconstrained binary quadratic programming} (UBQP) problems, and the subgradient descent method for solving the Lagrangian dual problem. In Sections \ref{sec: QSSP}, we focus on the QSS problem; we present the theoretical bounds on the penalty coefficients as well as the iterative methods for solving the QSS problem via solving its Lagrangian dual. Our experiments are described and numerical results are reported in Section \ref{sec: res}.  Finally, in Section \ref{sec:conclusion}, we state our concluding remarks.

\section{Preliminaries}\label{sec: litRvw}

\subsection{The quantum adiabatic approach to solving UBQP problems}\label{sec:prelim-aqc}

We refer the reader to \cite{Ronagh15} for a short introduction on quantum adiabatic computation. For a more extensive study, we refer the reader to \cite{farhi00} and \cite{farhi01} for the proposal of a quantum adiabatic algorithm by Farhi et al., and to \cite{Vazirani} for an exposition on its computational aspects.

These references suggest that practical quantum hardware can yield a significant quantum speedup in certain optimization problems. In particular, quantum annealers manufactured by D-Wave Systems Inc. solve a spin glass model problem where couplings connect pairs of quantum bits \cite{DwaveNature}. These annealers solve Ising models: 
\begin{equation} \label{Ising_def} \tag{Ising}
	\min_{s \in \mathbb \{-1,1\}^n}\, \sum_{(i,j)\in E} J_{ij} s_is_j + \sum_{i \in V} h_i s_i\,,
\end{equation}
where nonzero coefficients of the quadratic terms, that is, $E$, create a subgraph of a sparse graph structure known as the {\em Chimera} graph \cite{chimera}.

Note that any UBQP problem of the form
\begin{equation} \label{UBQP_def} \tag{UBQP}
	\min_{x \in \mathbb B^n}\, x^t\, Q\, x
\end{equation}
can be represented as an equivalent Ising model by using the affine transformation of $s_i = 2x_i -1$. Moreover, by using graph-minor embedding \cite{AidanRoy2014} and degree reduction techniques \cite{Ishikawa2011}, one might assume that these machines can solve any unconstrained binary polynomial programming problem. 
\begin{equation} \label{UBPP} \tag{UBPP}
	\max_{x \in \mathbb B^n}\, f(x)\,,
\end{equation}
where $f(x)$ is any polynomial in real coefficients. 

Our goal is to broaden the scope of application of quantum annealers by designing algorithms for solving constrained binary programming problems that work in conjunction with such oracles. 

In \cite{Ronagh15}, the authors proposed a method for solving CBQP problems using a branch-and-bound framework in which the bounding strategy is to solve the Lagrangian dual of the primal problem by successive application of quantum adiabatic evolution. The method described in \cite{Ronagh15} shows a fast rate of convergence to solution of the dual problem in every node of the branch-and-bound tree and provides a tight bound that drastically reduces the number of nodes traversed by the algorithm. However, it is important to mention that quantum annealers are coupled to an environment, and this significantly affects their performance. 

Albash et al. propose a noise model for D-Wave devices \cite{albash2014}. This model includes the control noise of the local field and couplings of the chip, as well as the effect of the cross-talk between qubits that are not coupled. In \cite{albash2014}, it is concluded that, despite the thermal excitations and small value of the ratio of the single-qubit decoherence time to the annealing time, an open-system quantum-dynamical description of the D-Wave device that starts from a quantized energy-level structure is well justified. The design of benchmark instances that can detect quantum speedup or any quantum advantage of a quantum annealer in comparison to state-of-the-art classical algorithms is studied by Katzgraber et al. \cite{katzgraber2015}. Zhu et al. \cite{zoshk2016} show that increasing the classical energy gap beyond the intrinsic noise level of the machine can improve the success of the D-Wave Two quantum annealer, at the cost of producing considerably easier benchmark instances. We refer the reader to \cite{king} for an explanation of the practicality of, and best practices in, using D-Wave devices.

\subsection{The Lagrangian dual problem} 

The Lagrangian dual (LD) of \eqref{BPP_def} is 
\begin{equation}\tag{LD}\label{def:LD}
 \min \limits_{\stack{\lambda \in \mathbb{R}^m}{\mu \in \mathbb R^p_-}}  \  d(\lambda, \mu),  
\end{equation}
where $d(\lambda, \mu)$ is evaluated via the \textit{Lagrangian relaxation}
\begin{equation}\tag{$\text{L}_{\lambda,\mu}$}\label{def: LR}
 d(\lambda, \mu) = \max \limits_{x\in \mathbb B^n} L(x,\lambda, \mu) =  f(x) + \lambda^t g (x) + \mu^t h(x) .
\end{equation}
Here, $g(x) = (g_1 (x) , \ldots, g_m(x))^t$ and $h(x) = (h_1(x) , \ldots, h_p (x))^t$ are the multivariable functions $g: \mathbb R^n \to \mathbb R^m$ and $h: \mathbb R^n \to \mathbb R^p$, respectively, with polynomial entries. The function $d(\lambda, \mu)$ is the maximum of a finite set of linear functions of $\lambda$ and $\mu$ and hence is convex and piecewise linear.

\begin{lem}[Weak Duality] The optimal value of \eqref{def:LD} is an upper bound for the optimal value of \eqref{BPP_def}. 
\end{lem}
\begin{proof}
Straightforward calculations show that given any fixed choice of $\lambda \in \mathbb{R}^m$ and $\mu \in \mathbb R^p_-$,
\begin{align*} 
v&:= \max_{x \in \mathbb B^n} \{ f(x) : g(x) = \mathbf 0, h(x) \leq \mathbf 0\} \\
&\leq \max_{x \in \mathbb B^n} \{ f(x) : \lambda^t  . g(x) = \mathbf 0, \mu^t . h(x) \geq \mathbf 0\} \\
&\leq \max_{x \in \mathbb B^n} \{ f(x) + \lambda^t . g(x) + \mu^t . h(x) \} \,.
\end{align*}
\end{proof}

\subsection{Subgradient method} 

Given any fixed tuple of Lagrange multipliers $(\lambda, \mu)$, the value of $d(\lambda, \mu)$ is the solution to a UBPP problem. This is the key fact in the method explained here.

To solve \eqref{def:LD}, an iterative approach may be employed. At the $k$-th iteration of the algorithm, $d(\lambda, \mu)$ is evaluated using a quantum annealing process, and a primal solution $x_k^\ast$ is attained. Note that $d(\lambda, \mu)$ is a convex function and the subgradient is a descent direction for $L$ because of which we can use a subgradient descent to the local (and hence global) minimum of $L$. The update rule for the multipliers in the $k$-th iteration of the algorithm will follow \cite{NonLinIntPrg_LiSun}:
\begin{equation}\label{GD-update}
\begin{array}{rl}
\lambda^{k+1} &= \lambda^k + s^k \frac{\nabla_\lambda L(x_k^\ast)}{\|\nabla_\lambda L(x_k^\ast)\|}\\
\mu^{k+1} &= P^- \left(\mu^k + s^k \frac{\nabla_\mu L(x_k^\ast)}{\|\nabla_\mu L(x_k^\ast)\|}\right),
\end{array} 
\end{equation}
where the projection $P^-$ keeps a $p$-dimensional vector $\mu$ in the negative orthant: 
$$P^- (\mu)= (\min (0, \mu_1), \ldots, \min (0, \mu_p)).$$
The algorithm is essentially \cite[Procedure 3.1]{NonLinIntPrg_LiSun}, where Step 1 is performed by quantum annealing. 

\begin{algorithm}
\caption{\, Quantum Gradient Descent}
\label{main}
\begin{tabbing} 
\quad \quad \= \quad \quad  \=\quad \=\quad \=\kill 
\> initialize: Lagrange multipliers $(\lambda^0, \mu^0)$ and $k= 0$ \\
\> {\bf until} $\text{termination}$ {\bf do} \\
\>\> solve \eqref{def: LR} using a quantum annealing device\\
\>\> find $(\lambda^{k+1}, \mu^{k+1})$ using update rule \eqref{GD-update}\\
\>\> $k \leftarrow k+1$
\end{tabbing}
\end{algorithm}

\noindent
$s^k$ in \eqref{GD-update} is the step size, and the choice of step-size can greatly affect the performance of the above algorithm. 

\section{The Quadratic Stable Set Problem}\label{sec: QSSP}
In this section, we investigate solving techniques for the {\em generalized quadratic stable set} problem defined as  
\begin{equation} \label{GQSS} \tag{GQSS}
\begin{array}{llcl}
	\max &  x^tWx,& \\
	\text{s.t.} & x^tAx=0,\\
	 & x \in \mathbb B^n,& 
\end{array}
\end{equation}
where $W\in \mathbb R^{n \times n}$ and $A\ge \mathbf 0$. Note that if we have several constraints of the form $x^t A^{(i)} x =0$ for $i\in I$ in \eqref{GQSS}, we can combine them into a single constraint $x^t \l( \sum_{i\in I} A^{(i)} \r) x =\mathbf 0$, hence reducing it to the form mentioned above. Moreover, without loss of generality, we may assume that all of the matrices in the quadratic terms in \eqref{GQSS} are symmetric, since $x^tQx = \frac{1}{2} x^t \l( Q+Q^t\r) x$ for any $Q \in \mathbb R^{n \times n}$. Finally, notice that for any $i,j$ where $A_{ij}\neq 0$, we have $x_i x_j =0$; therefore, we may pre-process \eqref{GQSS} such that $W\bullet A=\mathbf 0$ by setting $W_{ij}$ corresponding to nonzero $A_{ij}$ equal to zero, where $\bullet$ denotes the Hadamard (or entry-wise) product. The discussion above is summarized below in a list of assumptions on \eqref{GQSS}.

\newpage
{\bf Assumptions on \eqref{GQSS}:}
\begin{enumerate}
\item Problem \eqref{GQSS} has only a single constraint $x^tAx =0$, where $A\ge \mathbf 0$.
\item Matrices $A$ and $W$ are symmetric.
\item $W\bullet A =\mathbf 0$.
\end{enumerate}

A common technique for solving \eqref{GQSS} is penalizing the constraint in the objective function and solving the resulting unconstrained problem instead. It is easy to argue that for a sufficiently large penalty coefficient, $\lambda$, the UBQP problem 
\begin{equation}\label{UGQSS} \tag{GQSS$_\lambda$}
L(\lambda):= \max_{x \in \mathbb B^n} \ x^t W x - \lambda (x^t A x)\\
\end{equation} solves problem \eqref{GQSS}. 
However, when aiming to use a quantum annealer to solve this problem, it is important to choose the smallest possible $\lambda$ to increase the chance of attaining the optimal solution. The numerical experiment we present in the next section supports the idea that the smallest value of $\lambda$ for an arbitrary instance is generally much smaller than the theoretically derived penalty bound. This suggests substituting the penalty methods by an iterative scheme such as the subgradient method.

\subsection{Penalty methods for GQSS problems}\label{sec:GQSS-penalty-method}
Let $\mathcal N(i) := \{j \neq i: a_{ij} \neq 0\}$ and $W^+$ be the matrix containing non-negative entries of $W$. Note that by Assumption 3, $W_{ik}=0$ (hence $W^+_{ik}=0$) for all $k\in \mathcal N(i)$.
\begin{prop}\label{lmbBnd_prop}
Let $\lambda_i = \frac{ \frac{W^+_{ii}}{2} + \sum_{j \not \in \mathcal N(i)\neq i} W^+_{ij} } { \min_{j: A_{ij} \neq 0} A_{ij} }$. For any $\lambda > \tilde \lambda := \max_i \lambda_i$, formulation \eqref{UGQSS} solves \eqref{GQSS}.
\end{prop}
\begin{proof}
Our proof is by contradiction. Suppose $x^*$ is the optimal solution of \eqref{UGQSS}, but it is not feasible for \eqref{GQSS}; thus, $(x^*)^t A x^* \neq 0$, meaning that $\exists \  p, q$ such that $A_{pq}x^*_px^*_q >0$, that is, $A_{pq}\neq 0$ and $x^*_p, x^*_q = 1$. We argue that by setting $x^*_p=0$, we can improve the objective value; hence we reach the contradiction that $x^*$ was optimal for \eqref{UGQSS}.  Let $L(\lambda, x) := x^t W x - \lambda \left (x^t A x \right )$, and $\tilde x^\ast$ be the vector attained by setting $x_p^\ast=0$. It is easy to confirm that 
$$
\begin{array}{lll}
L(\lambda, \tilde x^\ast) & = & L(\lambda, x^\ast)  + 2 \left( -\frac{W_{pp}}{2} -  \sum_{j \not\in \mathcal N(p) \neq p }W_{pj} x_j \right)\\
& & + 2\lambda \l(A_{pq}+ \sum_{ j \in \mathcal N(p) \neq q } A_{pj} x_j \r)  \\
& \ge & L(\lambda, x^\ast)  + 2 \left( -\frac{W^+_{pp}}{2} -  \sum_{j \not\in \mathcal N(p) \neq p }W^+_{pj} x_j \right)+ 2\lambda A_{pq}\\
& \ge & L(\lambda, x^\ast)  + 2 \left( -\frac{W^+_{pp}}{2} -  \sum_{j \not\in \mathcal N(p) \neq p }W^+_{pj} x_j \right)+ 2\lambda \l({\min_{j: A_{ij} \neq 0} A_{ij}}\r)\\
&> &L(\lambda, x^\ast), 
\end{array}
$$ 
where the first inequality is satisfied by the fact that $ -W_{ij} \ge 0 $ for $ij$ not appearing in $W^+$ and $2\lambda \sum_{ j \in \mathcal N(p) \neq q } A_{pj} x_j \ge 0 $, and the second inequality is a result of our choice of $\lambda$, that is, 
$$
2\lambda \l({\min_{j: A_{ij} \neq 0} A_{ij}}\r) > W^+_{pp} +2 \sum_{j\not \in \mathcal N(p) \neq p} W^+_{pj}.
$$
\end{proof}

\begin{rmk} One may easily confirm that, in the absence of Assumption 3, $\lambda_i$ in Proposition \ref{lmbBnd_prop} can be modified as
$$\lambda_i = \frac{ \frac{W^+_{ii}}{2} + \sum_{j \not \in \mathcal N(i)\neq i} W^+_{ij}+ \max_{j \in \mathcal N(i)} W^+_{ij} } { \min_{j: A_{ij} \neq 0} A_{ij} }$$
so that the conclusion stays valid.
\end{rmk} 

\begin{rmk} 
The bound derived in Proposition \ref{lmbBnd_prop} is tight. Consider the following graph: 
\begin{center}
\begin{tikzpicture}[scale=0.5, shorten >=1pt, auto, node distance=3cm, thick,
   node_style/.style={circle,draw=black,font=\sffamily\bfseries},
   edge_style/.style={draw=black, ultra thick}]
\node[node_style] (1) at ( 90:4) {\footnotesize1};
\node[node_style] (2) at ( 162:4) {\footnotesize2};
\node [node_style](3) at (234:4) {\footnotesize3};
\node [node_style](4) at (306:4) {\footnotesize4};
\node [node_style] (5) at (18:4) {\footnotesize5};

\draw [above=2pt ](5) edge node{\tiny} (1)
[right](5) edge node{\tiny} (4);
\draw [very near start,above=0ex ](5) edge node{\tiny} (2)
[below](5) edge node{\tiny} (3);
\draw  [dotted, above] (2) edge node{$\tiny\omega$} (1);
\draw[dotted, below] (3) edge node{$\tiny\omega$} (4);
\draw [left=2pt](2) edge node{\tiny} (3);
\draw[dotted,left,very near start] (1) edge node{\tiny} (3) ;
\draw [dotted,very near start, right](1) edge node{\tiny}(4);
\draw[dotted, below,very near start] (2) edge node{\tiny} (4);
\path[dotted,every loop/.style={looseness=7}] (1)
         edge  [in=45,out=135,loop] node {\tiny} (); 
\path[dotted,every loop/.style={looseness=7}] (2)
         edge  [in=225-36,out=135-36,loop,left] node {\tiny} (); 
\path[dotted,every loop/.style={looseness=7}]
(3) edge  [in=225-36,out=315-36,loop] node {\tiny} ()
(4) edge  [in=225+36,out=315+36,loop] node {\tiny} (); 
\path[dotted,every loop/.style={looseness=7}]
(5) edge  [in=45+36,out=315+36,loop,right] node {\tiny} ();
\end{tikzpicture}
\end{center}
In the graph, solid and dotted lines stand for existing and non-existing edges, respectively; $A \in \mathbb B^{n \times n}$ is a binary matrix with entries $1$ corresponding to existing edges, and 0 otherwise; and weights matrix $W$ is zero everywhere except for pairs $\{ 1,2\}$ and $\{ 3,4\}$, on which the weight is $\omega$. Note that using Proposition \ref{lmbBnd_prop}, we have $\lambda > \omega$; however, if $\lambda \le \omega$, then  $x=\l( 1, 1, 1, 1 , 0\r)$ would be optimal for \eqref{UGQSS}, whereas it is infeasible for \eqref{GQSS}. 
\end{rmk}

In the hope of  decreasing  the penalty coefficient and improving the chance of observing the optimal solution, we use a separate $\lambda_{ij}$ for each nonzero entry of $A$, that is, solving  
\begin{equation}\label{LGQSS} \tag{GQSS$_\Lambda$}
L(\Lambda):= \max_{x \in \mathbb B^n} \ x^t W x - x^t (\Lambda \bullet A) x,\\
\end{equation} 
where $\Lambda \in \mathbb R^{n\times n}$, instead of \eqref{UGQSS}. We assume that $\Lambda$ is symmetric, similar to $A$. The following proposition, which is analogous to Proposition \ref{lmbBnd_prop}, shows how we can guarantee the solving of \eqref{GQSS} via \eqref{LGQSS} by the proper choice of $\Lambda$.  

\begin{prop} \label{LmbBnd_prop}
Let  $\lambda_i  = { \frac{w^+_{ii}}{2} + \sum_{j \not \in \mathcal N(i)\neq i} w^+_{ij} } $, where $\mathcal N(i)$ is as defined before. Problem \eqref{LGQSS} solves \eqref{GQSS} for matrix $\Lambda= [ \lambda_{ij}]$, where $\lambda_{ij}= \lambda_{ji} > \frac{ \max \{ \lambda_i, \lambda_j\} } {A_{ij} }$ for those $i$ and $j$ where $A_{ij} \neq 0$, and $\lambda_{ij} = 0$ otherwise.
\end{prop}
The proof of the above proposition is very similar to the proof of Proposition \ref{lmbBnd_prop}, and is omitted here to avoid repetition. Note that the largest entry of $\Lambda$ is equal to $\tilde \lambda$ from Proposition \ref{lmbBnd_prop}. 

Although the bounds of Propositions \ref{lmbBnd_prop} and \ref{LmbBnd_prop} are tight, when using the D-Wave 2X to solve \eqref{UGQSS} or \eqref{LGQSS}, we observed that these bounds fail to achieve the solution in many of our test cases. In the section that follows, we see that the subgradient method for solving the Lagrangian dual of \eqref{GQSS} can be viewed as a way of finding smaller penalty coefficients that are instance dependent and improve the performance of the D-Wave 2X in solving \eqref{GQSS}. 

\subsection{Iterative methods for the GQSS problem}\label{sec:GQSS-subgradient-method}

Note that by $A\ge \mathbf 0$,  we have $x^tAx \ge 0$ for all $x\in \mathbb B^n$. Therefore, we can substitute our equality constraint $x^t A x = 0$ with an inequality constraint as in the alternative formulation
\begin{equation}
\begin{array}{llc} \label{GQSSineq_def} \tag{$\overline{ \text{GQSS}}$}
\max & x^t W x\,,&\\
\text{s.t.}& x^t A x \le 0\,,\\
& x \in \mathbb B^n\,,&
\end{array}
\end{equation}
with Lagrangian dual
\begin{equation}\label{DGQSS}\tag{$\text D_{\overline{ \text{GQSS}}}$ }
\min_{\lambda \ge 0} \ \max_{x \in \mathbb B^n} \ x^t W x - \lambda (x^t A x)\,.
\end{equation}

\begin{prop}[Strong duality] An optimal dual solution of \eqref{DGQSS} corresponds to an optimal primal solution of \eqref{GQSS}. 
\end{prop} 

\begin{proof} 
In fact, a linear programming formulation of \ref{DGQSS} is 
\begin{equation}
\begin{array}{llcl}
	\min &  \delta\,,& \\
	\text{s.t.} & \delta \geq x^t W x - \lambda (x^t A x) &\forall x\in \mathbb B^n\,,\\
	& \lambda \geq 0, \delta \in \mathbb R\,.
\end{array}
\end{equation}
For every $x \in \mathbb B^n$, every cut $y= x^t W x - \lambda (x^t A x)$ has a non-positive slope. Since the constraint $x^t A x \leq 0$ is feasible (for example, at $x = \mathbf 0$), this dual problem is bounded and, at its optimal solution, the mentioned constraint is turned on. This proves strong duality. 
\end{proof}

Recall that the performance of Algorithm \ref{main} depends heavily on the step-size schedule. In this section, we look at a few iterative schemes to find a primal-dual solution $(x^\ast, \lambda^\ast)$ that solves \eqref{DGQSS}. Note that, since $x^t A x$ is always positive, the direction of the negative of the subgradient vector is always in the positive direction of the $\lambda$-axis.  Therefore, to solve \eqref{DGQSS}, we need only an increasing sequence $\{\lambda_k\}$. The subgradient method then always terminates in strong duality with a zero subgradient.

\subsubsection{Newtonian method}
This method is based on the update rule
\begin{equation}\label{lmbNew_def}
\lambda^{k}= \frac{ (x^{k-1})^t  W x^{k-1}} {(x^{k-1})^t A x^{k-1}},
\end{equation}
for all $k \geq 1$ and starting from $\lambda^0 = 0$.

\begin{algorithm}
\caption{}
\label{lmbNew_alg}
\begin{tabbing} 
\quad \quad \= \quad \quad  \=\quad \=\quad \=\kill 
\> initialize: $k= 0$, $\lambda^0=0$, and $x^{0}= \arg \max_{x\in \mathbb B^n} x^t W x$\\
\> {\bf while} $(x^k)^t A x^k \neq 0$ {\bf do} 
\\
\>\> $k=k+1$\\
\>\> find $\lambda^{k}$ using update rule \eqref{lmbNew_def}\\
\>\> let $x^{k}:= \arg \max_{x\in \mathbb B^n} x^t W x - \lambda^k (x^t A x)$
\end{tabbing}
\end{algorithm}

\begin{prop}
The sequence $\{\lambda^k \}$ generated by Algorithm \ref{lmbNew_alg} is an increasing sequence.
\end{prop}
\begin{proof}
Our proof is by induction. Note that $x^tWx=0$ for $x=0$. Therefore, $(x^0)^tWx^0 \ge 0$; as a result, $\lambda^1 \ge 0 = \lambda^0$.  By optimality of $x^{k}$, we have 
$$(x^{k})^t  W x^{k} - \lambda^k (x^{k})^t A x^{k} \ge (x^{k-1})^t W x^{k-1} - \lambda^k (x^{k-1})^t A x^{k-1}=0;$$ 
therefore,
$$\frac{(x^{k})^t  W x^{k}}{(x^{k})^t A x^{k}}- \lambda^{k}= \lambda^{k+1} - \lambda^k \ge 0.$$
\end{proof}

Notice that a feasible solution $x_f$ to \eqref{GQSSineq_def} provides a lower bound $x_f^t Wx_f$ for the Lagrangian dual of the problem. We can, therefore, modify the Newtonian method by replacing Equation \eqref{lmbNew_def} with 
\begin{equation}\label{lmbMdfNew_def}
\lambda^{k}= \frac{ (x^{k-1})^t  W x^{k-1}-(x_f)^t Wx_f} {(x^{k-1})^t A x^{k-1}}. 
\end{equation}
This suggests a modification of Algorithm \ref{lmbNew_alg}, which we refer to as the {\em {modified Newtonian method}}.  At iteration $k$, after obtaining a solution corresponding to $\lambda^k$, $x_i$'s are greedily set to zero until we reach a feasible solution, and the best feasible solution attained thus far is updated accordingly. The best feasible solution is then used in \eqref{lmbMdfNew_def} for finding the next $\lambda$.

Results from the next section suggest that the Newtonian method is an improvement over the theoretical bounds of Propositions \ref{lmbBnd_prop} and \ref{LmbBnd_prop}. The method, however, occasionally takes large steps, especially towards the end of the algorithm. To prevent this behaviour, for the remainder of this section, we suggest incrementing $\lambda$ with more-controlled step sizes.
\subsubsection{Incremental method} In the incremental method, the updating rule for $\lambda$ is\begin{equation}\label{lmbIncr_def}
\lambda^{k+1} = \lambda^k + \delta s,
\end{equation}
where $\delta \leq 1$ is a given constant. $\delta =1$ gives a fixed step-size update, and $\delta <1$ is a geometric update in which the step size shrinks as the algorithm proceeds. 

In utilizing a noisy quantum annealer, it is recommended to use the following termination criterion: the number of feasible solutions we wish to collect before termination is given to the algorithm (it is \verb+FeasCnt+ in Algorithm \ref{lmbIncr_alg}); after the termination, we pick the best observed feasible solution. Note that the iterates in Algorithm \ref{lmbNew_alg} cannot proceed after a feasible solution is obtained, because the denominator in \eqref{lmbNew_def} or \eqref{lmbMdfNew_def} is zero. As a result, to employ this termination criterion for Algorithm \ref{lmbNew_alg}, we need to update the iterates differently after reaching a feasible solution. One way to do this is to switch to an incremental scheme after observing a feasible solution. However, as argued earlier and supported by our experiments, the sequence of $\lambda$ could get exceedingly large in Algorithm \ref{lmbNew_alg}, and incrementing them afterwards does not result in a significant advantage.  

\begin{algorithm}
\caption{}
\label{lmbIncr_alg}
\begin{tabbing} 
\quad \quad \= \quad \quad  \=\quad \=\quad \=\kill 
\> given: \verb+ FeasCnt+ , $\delta\le1$, $\lambda^p$, and $s_\lambda$\\
\> initialize: $\lambda^0= \lambda^p$ and \verb+cnt+ $= 0$\\
\> {\bf for} $k= 1, 2, \ldots$ {\bf do} \\
\>\> $ \lambda^k = \lambda^{k-1} + s_\lambda$ \\
\>\> $ s_\lambda= \delta s_\lambda$\\
\>\> let $x^{k}:= \arg \max_{x\in \mathbb B^n} x^t W x - \lambda^k (x^t A x)$\\
\>\> {\bf if} $(x^{k})^t A x^{k} = 0$\\
\>\> \> \verb+ cnt +$=$\verb+ cnt+ $+1$\\
\>\>{\bf if} \verb+ cnt+ $\ge$ \verb+FeasCnt+ \\
\>\>\>\> terminate
\end{tabbing}
\end{algorithm}

The incremental method avoids the problem of taking long jumps, unlike the Newtonian method. The shortcoming is that it is highly dependent on $s_\lambda$ and may take many iterations if $\lambda_p$ and/or $s_\lambda$ are too small. In what follows, we propose a hybrid technique that combines the advantages of both the Newtonian and incremental methods. 

\subsubsection{Hybrid method} \label{sec:hybrd} In the hybrid method, the step sizes are proportionate to the length of the subgradients. Suppose that at iteration $k$ we have $\lambda^k$ and 
\begin{equation}
 x^k := \arg\max_{x\in \mathbb B^n}  \  x^t W x - \lambda^k \l(x^t A x\r). 
 \end{equation}
Then, 
\begin{equation} 
L(\lambda) \geq (x^k)^t W x^k - \lambda (x^k)^t A x^k,
\end{equation}
and the gradient of the right-hand side, $\delta^k= (x^k)^t A x^k$, is a subgradient for $L$.  In our hybrid scheme, we set 
\begin{equation}
\lambda^{k+1}=\lambda^{k} + s_\lambda^k, \quad \text{for} \quad  
s_\lambda^k = \tilde \alpha \delta^k\,.
\end{equation}

Similar to the incremental method, in the hybrid method we wish to collect a certain number of feasible solutions before termination. Because $\delta^k$ vanishes when we see a feasible solution, we switch to the incremental method with $\lambda^p = \lambda^k$ and a given $s_\lambda$ after reaching the first feasible solution. The complete algorithm is given below. 
\begin{algorithm}
\caption{}
\label{lmbHybrid_alg}
\begin{tabbing} 
\quad \quad \= \quad \quad  \=\quad \=\quad \=\kill 
\> given: $s_\lambda$ and \verb+FeasCnt+\\
\> initialize: $k= 0$, $\lambda^0=0$, $x^{0}= \arg \max_{x\in \mathbb B^n} x^t W x$, and $\tilde \alpha$\\
\> {\bf while} $ (x^k)^t A x^k \neq 0$ {\bf do} \\
\>\> $k\leftarrow k+1$\\
\>\> $\delta^{k-1}= (x^{k-1})^t A x^{k-1} $\\
\>\> $ \lambda^{k} = \lambda^{k-1} + \tilde \alpha \delta^{k-1}$ \\
\>\> $x^{k}= \arg \max_{x\in \mathbb B^n} x^t W x - \lambda^k (x^t A x)$\\
\> {\bf go to} Algorithm \ref{lmbIncr_alg} with $\lambda^p:=\lambda^k$, $s_\lambda$, $\delta=1$, and \verb+FeasCnt+
\end{tabbing}
\end{algorithm}

\section{Experimental results} \label{sec: res}

In this section, we report our experimental results of solving \eqref{GQSS} using both the penalty methods of Section \ref{sec:GQSS-penalty-method} and the iterative methods of Section \ref{sec:GQSS-subgradient-method}. The test instances of \eqref{GQSS} were generated randomly, with matrices $W$ taking integer entries between $-5$ and $5$, and $A$ generated as sparse binary matrices with a sparsity of $0.4$. 

Let us first motivate our discussion by experimentally showing the inconvenience caused by large penalty coefficients. The impact of $\lambda$ on the chance of observing the optimal solution from the quantum annealer is depicted in Figure \ref{lmbBndFail_fig} for two random instances. In the plots of Figure \ref{lmbBndFail_fig}, the range of $\lambda$, that is, from 0 to the theoretical bound for $\lambda$ derived by Proposition \ref{lmbBnd_prop}, is divided into 100 steps, and the bound of Proposition \ref{lmbBnd_prop} is depicted with a dashed line. For each value of $\lambda$ for which the quantum annealer succeeded in reaching the optimal solution, there is a `$+$' mark showing the number of times the optimal solution was observed among 1000 reads. Note that in both cases, the values of $\lambda$ for which the optimal solution is observed are smaller than the theoretical bound, and, the smaller $\lambda$ is, the higher is the chance of  observing the optimal solution.
\begin{figure}[H]
\centering     
\subfigure[]{\label{lmbBndFail_fig:a}\includegraphics[scale=0.38]{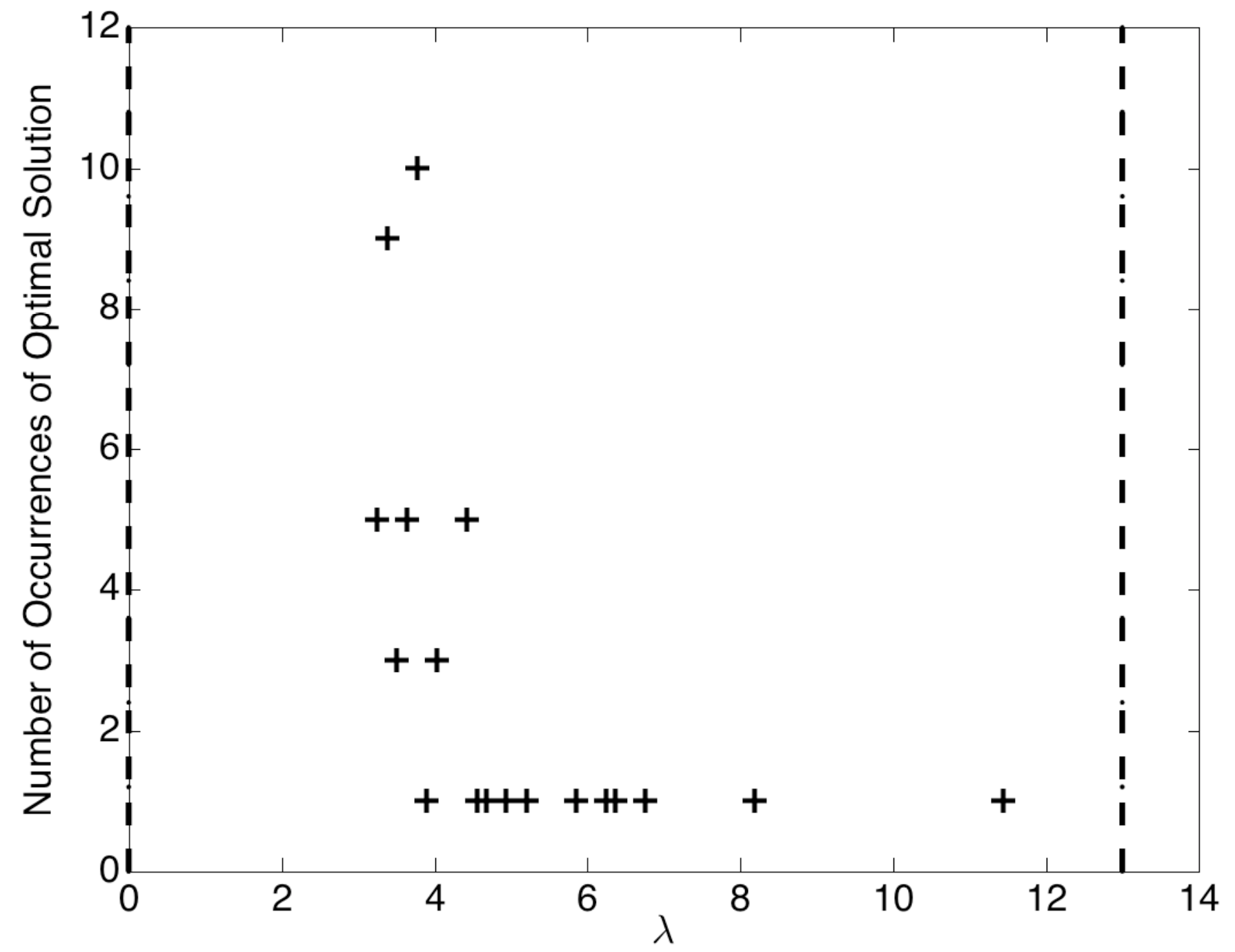}} \quad
\subfigure[]{\label{lmbBndFail_fig:b}\includegraphics[scale=0.38]{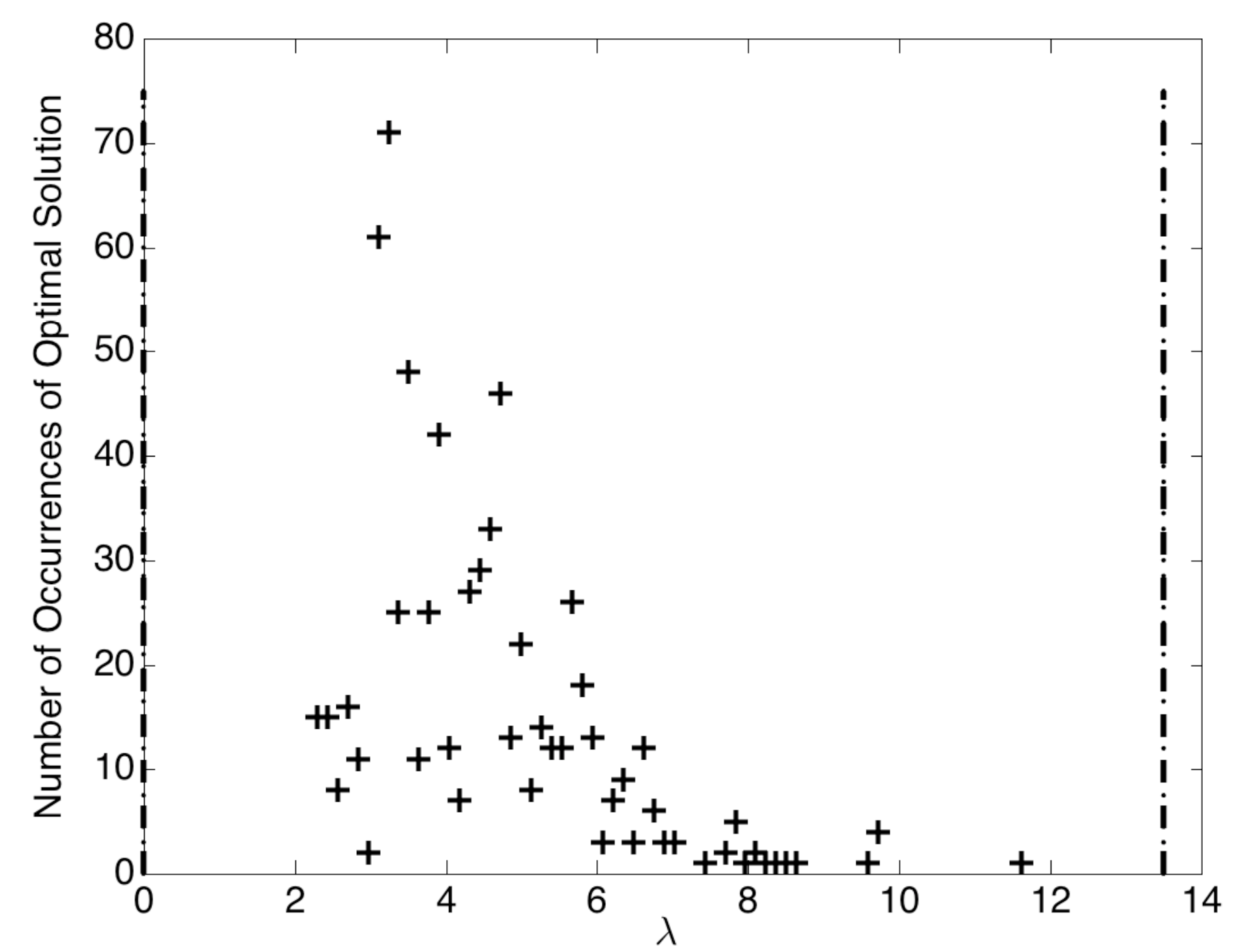}}
\vspace{-0.5cm}
\caption{Impact of $\lambda$ on observing the optimal solution.}
\label{lmbBndFail_fig}
\end{figure}

We generated 30 random test instances of the type described above. We ran two experiments on them, one using the D-Wave 2X, with queries handled by D-Wave's SAPI 2.2.1, and another using single-flip simulated quantum annealing (SQA), based on the Trotterization of the path integral of the quantum Hamiltonian of the Ising spin model with transverse field as discussed in \cite{HRIT, SantoroTosatti}.

In the experiment of solving the test instances using the D-Wave 2X, the instances were embedded on the Chimera graph of 1100 working qubits using the {\sf findEmbedding} functionality of SAPI 2.2.1. Then, the {\sf sapiSolveIsing} function was used with solutions post-processed in the {\sf optimization} mode. 
The ferromagnetic chain strengths were set to $-1$ after scaling the objective function such that all coefficients were between $-1$ and $1$. The chain strengths were iteratively incremented by the geometric series with a ratio of $0.8$. This iteration was terminated after {\sf sapiUnembedAnswer} collected $10\%$ of a set number of {\sf numreads} in {\sf discard} mode. 

We have tested the theoretical bounds of Propositions \ref{lmbBnd_prop} and \ref{LmbBnd_prop}, as well as the four iterative schemes of Section \ref{sec:GQSS-subgradient-method}---the Newtonian method, the modified Newtonian method, the incremental method, and the hybrid method---on our randomly generated test cases. The Newtonian and modified Newtonian methods were terminated as soon as the first feasible solution was obtained. The incremental method was initialized with $\lambda_0 =0$, $\delta =1$, and $s_\lambda =1$ (i.e., we used fixed step-size increments), and we collected five feasible solutions before termination (i.e., \verb+FeasCnt+ $ =5$). 
For the hybrid method, we found $\tilde \alpha$ such that $\lambda^1$ coincides with $\lambda^1$ from the Newtonian method, that is, 
$$\tilde \alpha \delta^0 =\frac{(x^0)^t W x^0}{(x^0)^t A x^0},$$ 
which concludes
 $$\tilde \alpha (x^0)^t A x^0  = \frac{(x^0)^t W x^0}{(x^0)^t A x^0} \Rightarrow  \tilde \alpha = \frac{(x^0)^t W x^0  }{\l((x^0)^t A x^0\r)^2},$$
where $x^0 = \arg\max_{x \in \mathbb B^n} x^t Wx $. To avoid taking extremely small step sizes, we bounded our step-size scheme by 0.05, so in our experiment we took $\tilde \alpha = \max \l(  \frac{(x^0)^t W x^0  }{\l((x^0)^t A x^0\r)^2} , 0.05 \r)$. Similar to the incremental method, we collected five feasible points before termination, so \verb+FeasCnt+ $=5$. In our experiment, we set $s_\lambda=0.5$. 

Table \ref{res_G} summarizes the results of solving our test cases using the D-Wave 2X. In this table, {$\mathsf n$} is the dimension of the problem, and {\sf opt} is the optimal objective value of the problem found using  Gurobi Optimizer 6.0.4 \cite{gurobi}. The results for theoretical bounds attained in Propositions \ref{lmbBnd_prop} and \ref{LmbBnd_prop} are presented in columns {\sf prop1} and {\sf prop2}, respectively. The columns {\sf new}, {\sf m-new}, {\sf incr}, and {\sf hyb} refer to the Newtonian, modified Newtonian, incremental, and hybrid methods, respectively. For each method, the best objective value {\sf obj} and the corresponding Lagrangian multiplier $\lambda$ are reported. In {\sf prop2}, $\lambda$ refers to the range of the entries in the matrix of Lagrangian multipliers. For each instance, an asterisk indicates whether the optimal solution was achieved. The column {\sf cnt} in each of the iterative methods, {\sf new}, {\sf m-new}, {\sf incr}, and {\sf hyb}, is the total number of reads over all iterations of each method, obtained from the D-Wave 2X and reported in thousands. The number of reads (again reported in thousands) for {\sf prop1} and {\sf prop2} were then selected to exceed the largest number of reads performed by all of the iterative methods in favour of the penalty methods. It is worth recalling that each iteration of any of our methods, including {\sf prop1} and {\sf prop2}, may have several calls to the D-Wave 2X while chain strengths increase until we collect at least $10\%$ of a set number of {\sf numreads} in {\sf discard} mode. Therefore, each iteration may have a different total {\sf numreads} value, because of which the columns {\sf cnt} in {\sf prop1} and {\sf prop2} are not equal. 

To dissociate the effect of embedding on our results, we also conducted an experiment with SQA on our test cases; the results of this experiment are summarized in Table \ref{res_N}. In the SQA queries, the inverse temperature $\beta$ was set to 15. The number of Trotter slices was set to 15. The strength of the transverse field was linearly scheduled to decrease from 3 to 0.1. Each read from SQA consisted of 100 sweeps through the effective Hamiltonian in one dimension higher. In Table \ref{res_N}, {\sf cnt} is the total number of reads for a complete run of each method; the number of reads in each iteration of the iterative methods was 20. The {\sf INF} for the objective function stands for \emph{infeasible} and refers to instances where SQA did not return a feasible solution.

\begin{table} \caption{Experimental results using the D-Wave 2X.}
\footnotesize
\begin{threeparttable}
\begin{tabular}{| ll | lll | lll | lll | lll | lll | lll |} 
\hline
& &\multicolumn{3}{c|}{\sf prop1} & \multicolumn{3}{c|}{\sf prop2} & \multicolumn{3}{c|}{\sf new}& \multicolumn{3}{c|}{\sf m-new}& \multicolumn{3}{c|}{\sf incr}& \multicolumn{3}{c|}{\sf hyb} \\
\cline{3-20}
{\sf n} & {\sf opt} & $\lambda$ & {\sf obj}& {\sf cnt} &$\lambda$ & {\sf obj} & {\sf cnt} & $\lambda$ & {\sf obj}& {\sf cnt} & $\lambda$ & {\sf obj}& {\sf cnt} & $\lambda$ & {\sf obj}& {\sf cnt} & $\lambda$ & {\sf obj}& {\sf cnt} \\
\hline 
30 & 61 & 	29 & 50 &40 & 9-29 & 33 & 40 & 6.5 & 61$^\ast$ & 24 & 5.5 & 55 & 27 & 5 & 61$^\ast$ & 34 & 3.3 & 61$^\ast$ & 28\\
30 & 75 & 32 & 66 & 80 & 8-32 & 66 & 70& 9.4 & 75$^\ast$ & 22 & 5.5 & 75$^\ast$ & 26 & 6 & 75$^\ast$ & 38 & 5.6 & 75$^\ast$ & 60\\
30 & 54 & 	32 & 	48 &170 & 9-32 & 46&180 & 14.3 & 40 & 38 & 7.5 & 47 & 51 & 8 & 48 & 74 & 7.4 & 47 & 166 \\
30 & 64 & 	35 & 	41 & 50 & 	13-35 & 52 & 60 & 36 & 32 & 37 & 5.2 & 56 & 42 & 5 & 64$^\ast$ & 46 & 3.4 & 	64$^\ast$ & 32\\
30 & 56 & 	34 & 	46 & 60 & 	11-34 & 44 & 60 & 32.5 & 38 & 25 & 5.5 & 56$^\ast$ & 33 & 6 & 56$^\ast$ & 37 & 4 & 56$^\ast$ & 41\\
30 & 49 & 	35 & 	32 & 60 & 	6-35 & 44 & 70 & 16 & 32 & 37 & 4 & 49$^\ast$ & 36 & 4 & 49$^\ast$ & 44 & 3.7 & 49$^\ast$ & 43\\
30 & 48 & 	32 & 43 & 80 & 6-32 & 44 & 80 & 16.7 & 34 & 26 & 9 & 44 & 34 & 6 & 46 & 49 &	4.8 & 48$^\ast$ & 70\\
30 & 68 & 	33 & 	54 &120 & 6-33 & 45 & 110 & 5 & 63 & 38 & 6.1 & 61 & 49 & 	6 & 68$^\ast$ & 58 & 5.5 & 68$^\ast$ & 84\\
30 & 92 & 	33 & 	76 & 60 & 	5-33 & 68 & 80 & 5.1 & 92$^\ast$ & 24 & 5.3 & 92$^\ast$ & 32 & 5 & 92$^\ast$ & 38 & 4.6 & 92$^\ast$ & 29\\
30 & 65 & 	35 & 60 & 80 & 	7-35 & 59 & 80 & 38 & 47 & 29 & 7.8 & 65$^\ast$ & 30 & 7 & 65$^\ast$ &45 & 5.1 & 65$^\ast$ & 62\\
30 & 99 & 	39 & 	89 &120 & 11-39 & 89 & 120 & 	56.5 & 41 & 27 & 7.5 & 99$^\ast$ & 36 & 9 & 99$^\ast$ & 46 & 	7 & 99$^\ast$ & 95\\
30 & 57 & 	36 & 41 & 120 & 9-36 & 41 & 120 & 16.7 & 40 & 32 & 9.2 & 40 & 36 & 6 & 49 & 48 & 8.2 & 	57$^\ast$ & 88\\
30 & 56 & 	31 & 44 & 60 & 	5-31 & 54 & 60 & 8.9 & 48 & 17 & 6.3 & 56$^\ast$ & 24 & 	9 & 52 & 34 & 5.7 & 56$^\ast$ & 45\\
30 & 64 & 	31 & 54 & 180 & 10-31 & 54 & 150 & 40 & 54 & 38 & 7.5 & 61 & 54 & 8 & 	63 & 74 & 	5.1 & 64$^\ast$ & 125\\
30 & 61 & 	33 & 55	& 120 & 11-33 & 55 & 160 & 41 & 43 & 25 & 11.5 & 55 & 25 & 8 & 61$^\ast$ & 42 & 6 & 61$^\ast$ & 115 \\
30 & 52 & 	32 & 48 &150 & 7-32 & 49 & 130 & 32 & 42 & 41 & 8 & 52$^\ast$ & 48 & 6 & 52$^\ast$ & 59 & 5.1 & 52$^\ast$ & 122\\
30 & 69 & 	40 & 	42 &180 & 9-40 & 60 &160 & 6.9 & 62 & 38 & 10 & 46 & 43 & 7 & 69$^\ast$&62 & 4.9 & 69$^\ast$ & 138\\
30 & 74 & 35 & 45 & 160 & 14-35 & 49 &100 & 21.7 & 39 & 53 & 5.2 & 74$^\ast$ & 75 & 3 & 74$^\ast$ & 86 & 2.8 & 74$^\ast$ & 55\\
30 & 64 & 	33 & 	59&120 & 	7-33 & 61 & 100 & 39.5 & 30 & 25 & 10 & 64$^\ast$ & 25 & 7 & 64$^\ast$ & 39 & 5.7 & 64$^\ast$ & 97\\
30 & 80 & 	38 & 67 & 80 & 6-38 & 80$^\ast$ & 80 & 8.2 & 	80$^\ast$ & 26 & 10 & 80$^\ast$ & 33 & 6 & 80$^\ast$ & 44 & 5.5 & 80$^\ast$ & 52\\
30 & 63 & 	37 & 59 &120 & 11-37 & 62 &120 & 24 & 48 & 25 & 10.8 & 60 & 27 & 8 & 63$^\ast$ & 49 & 7.7 & 63$^\ast$ &106\\
30 & 	58 & 	32 & 55 & 90 & 	11-32 & 50 & 100 & 20.2 & 30 & 32 & 6 & 51 & 35 & 	9 & 55 & 48 & 6.3 & 57 & 89\\
30 & 	64 & 	32 & 	60 & 80 & 9-32 & 53 & 80 & 7.5 & 60 & 29 & 6 & 64$^\ast$ & 41 & 5 & 64$^\ast$ & 51& 3.9 & 64$^\ast$ & 69\\
30 & 	57 & 	33 & 57$^\ast$ & 120 & 8-33 & 57$^\ast$ & 120 & 9.5 & 57$^\ast$ & 27 & 10 & 57$^\ast$ & 29 & 7 & 57$^\ast$ & 50 & 6.5 & 57$^\ast$ &120\\
30 & 	50 & 	29 & 	38 & 90 & 	8-29 & 48 & 100 & 	31.5 & 33 & 26 & 5.5 & 48 & 38 & 7 & 49 & 39 & 5.2 & 49 & 77\\
30 & 	53 & 33 & 47 & 80 & 6-33 & 47 & 60 & 14.1 & 46 & 32 & 7.1 & 53$^\ast$ & 32 & 	5 & 53$^\ast$ & 40 & 4.9 & 53$^\ast$ & 55\\
30 & 	76 & 	34 & 69 & 140 & 9-34 & 76$^\ast$ & 120 & 25.7 & 53 & 31 & 7.7 & 76$^\ast$ & 32 & 7 & 76$^\ast$ & 56 &	7.1 & 76$^\ast$ & 115\\
30 & 	71 & 	35 & 57 &140 & 5-35 & 63 & 110 & 24.5 & 45 & 40 & 9 & 69 & 36 & 8 & 69 & 57 & 6.5 & 69 & 89\\
30 & 	52 & 	30 & 	40 & 80 & 9-30 & 50&60 & 8.8 & 47 & 23 & 7 & 49 & 26 & 8 & 52$^\ast$ & 32 & 5.3 & 52$^\ast$ & 46\\
30 & 	50 & 	28 & 43 & 60 & 7-28 & 41& 60 & 6.2 & 50$^\ast$ & 22 & 5.2 & 50$^\ast$ & 49 & 5 & 50$^\ast$ & 32 & 4.8 & 50$^\ast$ & 43\\
\hline
\end{tabular}
 \begin{tablenotes}
      \item[$\dagger$] \footnotesize{ The columns {\sf cnt} are reported in thousands.}
    \end{tablenotes}
\label{res_G}
\end{threeparttable}
\end{table}
\begin{table} \caption{Experimental results using SQA.}
\footnotesize
\begin{threeparttable}
\begin{tabular}{| ll | lll | lll | lll | lll | lll | lll |} 
\hline
& &\multicolumn{3}{c|}{\sf prop1} & \multicolumn{3}{c|}{\sf prop2} & \multicolumn{3}{c|}{\sf new}& \multicolumn{3}{c|}{\sf m-new}& \multicolumn{3}{c|}{\sf incr}& \multicolumn{3}{c|}{\sf hyb} \\
\cline{3-20}
{\sf n} & {\sf opt} & $\lambda$ & {\sf obj}& {\sf cnt} &$\lambda$ & {\sf obj} & {\sf cnt} & $\lambda$ & {\sf obj}& {\sf cnt} & $\lambda$ & {\sf obj}& {\sf cnt} & $\lambda$ & {\sf obj}& {\sf cnt} & $\lambda$ & {\sf obj}& {\sf cnt} \\
\hline 
30 & 61 & 29 & 45 & 200 & 9-29 & 58 & 200 & 5 & 58 & 60 & 6 & 55 & 200 & 7 & 61$^\ast$ & 180 & 5.31 & 61$^\ast$ & 200 \\  
30 & 75 & 32 & 75 & 220 & 8-32 & 53 & 220 & 24.25 & 59 & 60 & 6 & 75$^\ast$ & 80 & 6 & 75$^\ast$ & 200 & 5.7 & 75$^\ast$ & 220 \\  
30 & 54 & 32 & 42 & 340 & 9-32 & 40 & 340 & 16.17 & 39 & 80 & 7.5 & 47 & 100 & 8 & 48 & 260 & 8.01 & 54$^\ast$ & 340\\  
30 & 64 & 35 & 56 & 180 & 13-35 & {\sf INF} & 180 & 11.25 & 64$^\ast$ & 80 & 6 & 64$^\ast$ & 80 & 5 & 64$^\ast$ & 180 & 5.64 & 64$^\ast$ & 140 \\  
30 & 56 & 34 & 56 & 200 & 11-34 & 43 & 200 & 4.96 & 56$^\ast$ & 60 & 9 & 49 & 200 & 5 & 56$^\ast$ & 200 & 4.43 & 56$^\ast$ & 140 \\  
30 & 49 & 35 & 47 & 160 & 6-35 & 33 & 160 & 16 & 49$^\ast$ & 80 & 4.75 & 49$^\ast$ & 120 & 4 & 49$^\ast$ & 160 & 4.79 & 49$^\ast$ & 140 \\  
30 & 48 & 32 & 45 & 200 & 6-32 & 41 & 200 & 30.5 & 36 & 80 & 5.75 & 48$^\ast$ & 80 & 6 & 48$^\ast$ & 200 & 5.75 & 48$^\ast$ & 180 \\  
30 & 68 & 33 & 51 & 360 & 6-33 & 68$^\ast$ & 360 & 7.83 & 63 & 60 & 7 & 68$^\ast$ & 280 & 8 & 68$^\ast$ & 240 & 7.35 & 68$^\ast$ & 360 \\  
30 & 92 & 33 & 59 & 180 & 5-33 & 52 & 180 & 14 & 74 & 160 & 9.3 & 92$^\ast$ & 60 & 5 & 92$^\ast$ & 180 & 6.33 & 92$^\ast$ & 140 \\  
30 & 65 & 35 & 60 & 300 & 7-35 & 62 & 300 & 16 & 56 & 160 & 7.83 & 60 & 60 & 6 & 65$^\ast$ & 200 & 6.66 & 65$^\ast$ & 300 \\  
30 & 99 & 39 & 90 & 300 & 11-39 & 90 & 300 & 56.5 & 62 & 80 & 7 & 99$^\ast$ & 100 & 7 & 99$^\ast$ & 220 & 9.37 & 99$^\ast$ & 300 \\  
30 & 57 & 36 & 49 & 280 & 9-36 & 40 & 280 & 20.75 & 46 & 80 & 7.75 & 57$^\ast$ & 60 & 8 & 57$^\ast$ & 260 & 7.04 & 57$^\ast$ & 280\\  
30 & 56 & 31 & 49 & 300 & 5-31 & 42 & 300 & 15.67 & 35 & 60 & 6.5 & 56$^\ast$ & 80 & 7 & 56$^\ast$ & 220 & 6.54 & 56$^\ast$ & 300 \\
30 & 64 & 31 & 61 & 260 & 10-31 & 63 & 260 & 22 & 57 & 80 & 9 & 64$^\ast$ & 100 & 9 & 64$^\ast$ & 240 & 6.82 & 64$^\ast$ & 260 \\
30 & 61 & 33 & 37 & 300 & 11-33 & 59 & 300 & 9.42 & 59 & 60 & 11.5 & 59 & 60 & 12 & 59 & 280 & 7.46 & 61$^\ast$ & 300 \\
30 & 52 & 32 & 49 & 220 & 7-32 & 48 & 220 & 32 & 34 & 80 & 6 & 52$^\ast$ & 60 & 6 & 52$^\ast$ & 200 & 9.05 & 52$^\ast$ & 220 \\
30 & 69 & 40 & 69$^\ast$ & 360 & 9-40 & 69$^\ast$ & 360 & 6.94 & 69$^\ast$ & 60 & 6.67 & 69$^\ast$ & 60 & 7 & 69$^\ast$ & 320 & 5.47 & 69$^\ast$ & 360 \\
30 & 74 & 35 & 53 & 160 & 14-35 & {\sf INF} & 160 & 7.69 & 74$^\ast$ & 60 & 3.82 & 74$^\ast$ & 40 & 4 & 74$^\ast$ & 160 & 3.87 & 74$^\ast$ & 160 \\
30 & 64 & 33 & 64$^\ast$ & 360 & 7-33 & 62 & 360 & 39.5 & 47 & 80 & 10 & 59 & 60 & 7 & 64$^\ast$ & 240 & 8.97 & 64$^\ast$ & 360 \\
30 & 80 & 38 & 67 & 180 & 6-38 & 62 & 180 & 8.2 & 80$^\ast$ & 60 & 10 & 71 & 80 & 5 & 80$^\ast$ & 180 & 6.15 & 80$^\ast$ & 160 \\
30 & 63 & 37 & 63$^\ast$ & 360 & 11-37 & 59 & 360 & 40 & 63$^\ast$ & 80 & 15.5 & 58 & 80 & 9 & 63$^\ast$ & 260 & 8.1 & 63$^\ast$ & 360 \\
30 & 58 & 32 & 55 & 260 & 11-32 & 51 & 260 & 20.75 & 42 & 80 & 7.5 & 58$^\ast$ & 100 & 8 & 57 & 260 & 8.71 & 58$^\ast$ & 240 \\
30 & 64 & 32 & 60 & 340 & 9-32 & 46 & 340 & 7.5 & 64$^\ast$ & 140 & 6 & 64$^\ast$ & 220 & 6 & 64$^\ast$ & 220 & 6.19 & 64$^\ast$ & 340 \\
30 & 57 & 33 & 43 & 440 & 8-33 & {\sf INF}  & 440 & 37 & 36 & 80 & 9.75 & 54 & 160 & 11 & 54 & 300 & 11.1 & 57$^\ast$ & 440 \\
30 & 50 & 29 & 48 & 240 & 8-29 & 44 & 240 & 5.59 & 49 & 60 & 7.5 & 50$^\ast$ & 300 & 12 & 39 & 240 & 8.21 & 50$^\ast$ & 240 \\
30 & 53 & 33 & 43 & 220 & 6-33 & {\sf INF}  & 220 & 11.88 & 53$^\ast$ & 80 & 6.63 & 53$^\ast$ & 60 & 6 & 53$^\ast$ & 200 & 6.14 & 53$^\ast$ & 220 \\
30 & 76 & 34 & 61 & 260 & 9-34 & 76$^\ast$ & 260 & 25.75 & 33 & 80 & 6.83 & 76$^\ast$ & 60 & 7 & 76$^\ast$ & 220 & 7.03 & 76$^\ast$ & 260 \\
30 & 71 & 35 & 63 & 340 & 5-35 & 65 & 340 & 26.25 & 60 & 120 & 9.38 & 71$^\ast$ & 60 & 11 & 71$^\ast$ & 260 & 8.76 & 71$^\ast$ & 340 \\
30 & 52 & 30 & 48 & 240 & 9-30 & 52$^\ast$ & 240 & 8.83 & 52$^\ast$ & 60 & 8 & 52$^\ast$ & 100 & 7 & 52$^\ast$ & 220 & 7.59 & 52$^\ast$ & 240 \\
30 & 50 & 28 & 41 & 220 & 7-28 & 41 & 220 & 11.25 & 38 & 100 & 6.83 & 50$^\ast$ & 160 & 6 & 50$^\ast$ & 220 & 5.91 & 50$^\ast$ & 220 \\
\hline
\end{tabular}
\label{res_N}
\end{threeparttable}
\end{table}


\section{Discussion}\label{sec:conclusion}
We proposed an iterative scheme for solving the Lagrangian dual of a constrained binary programming problem using a quantum annealer.  We have tested several settings of our iterative method for a specific class of constrained binary programming problem, namely, the generalized quadratic stable set (GQSS) problem.
Our results show that the iterative methods outperformed the theoretical bounds for this problem. The hybrid method, in particular, performed the best, and could achieve the optimal solution in most of our test cases. Another important observation from the results is that the modified Newtonian method overcomes the overshooting problem of the step-size schedule in the Newtonian method. As the algorithm proceeds, $x^t A x$ approaches 0, and $x^tWx$ may get larger, so the ratio $\frac{x^t W x}{x^t A x}$ can jump to large values. Whereas this happens for the Newtonian method, the modified variant of this method avoids this problem because the best objective value observed for a feasible solution $x_f$ also improves, so $x^t W x - x_f W x_f$ stays minimal.

Since the Chimera graph is very sparse, problem instances in which the connectivity of $W+A$ is a subgraph of the Chimera graph are similarly sparse and often disconnected. Therefore, random test instances of the GQSS problem that are native to the Chimera graph are not sufficiently difficult for the purpose of our experiment.

It is worth mentioning that the performance of the outer approximation method of \cite{Ronagh15} for the GQSS problem is similar to the penalty methods. In the outer approximation method, a linear programming (LP) problem is initialized with a set of sufficiently large bounds (box constraints) on the Lagrangian multipliers; iteratively, the solution of the LP problem is employed to form \eqref{def: LR}, and, based on the solution of \eqref{def: LR}, a linear constraint (cut) is added to the LP problem; this procedure is terminated when no new cuts are generated. For the GQSS problem, to ensure that outer approximation returns the Lagrangian dual bound (instead of a looser lower bound), the box constraint must be as large as the theoretical penalty value. The Lagrangian relaxation with this multiplier is the first (and only) slave unconstrained problem solved, and, consequently, the success of this method is the same as that of the penalty methods. Therefore, \cite{Ronagh15} does not overcome the deficiencies of the penalty methods when working with noisy quantum annealers.  

Strong duality holds for the GQSS problem; thus, the subgradient method may be viewed as a technique for finding smaller penalty coefficients that are instance dependent. In general, strong duality does not hold for a constrained binary quadratic problem. Similar to the idea presented in \cite{Ronagh15}, we may employ the subgradient descent method within a branch-and-bound framework as a bounding procedure to solve a general constrained binary quadratic programming problem. However, missing the optimal solution can occasionally break the branch-and-bound framework by obtaining an incorrect bound based on a suboptimal solution. A suitable recovery scheme, or a guarantee for checking the optimality of a solution, is needed to successfully employ any of these methods in a branch-and-bound framework. This nontrivial task is a subject for future study.

\end{document}